\documentclass[9pt, journal]{IEEEtran}
\usepackage{amssymb, bm, mathrsfs, color, subcaption,booktabs}

\newtheorem{theorem}{Theorem}[section]

\newtheorem{definition}[theorem]{Definition}

\newtheorem{assumption}[theorem]{Assumption}

\newtheorem{lemma}[theorem]{Lemma}
\newtheorem{remark}[theorem]{Remark}

\newcommand{\argmax}{\mathop{\rm arg\hspace{2pt}max}}
\newcommand{\vertiii}[1]{{\vert\kern-0.25ex\vert\kern-0.25ex\vert #1 
		\vert\kern-0.25ex\vert\kern-0.25ex\vert}}
\newcommand{\bigvertiii}[1]{{\big\vert\kern-0.35ex\big\vert\kern-0.35ex\big\vert #1 
		\big\vert\kern-0.35ex\big\vert\kern-0.35ex\big\vert}}

\usepackage{cite}

\ifCLASSINFOpdf
\usepackage[pdftex]{graphicx}
\else
\fi
\usepackage[cmex10]{amsmath}
\usepackage{array}
\usepackage{url}

\begin{document}
	%
	\title{Self-triggered Stabilization of Discrete-time Linear
		Systems with Quantized State Measurements
	}
	\author{Masashi~Wakaiki,~\IEEEmembership{Member,~IEEE}
		\thanks{
			M.~Wakaiki is with the 
			Graduate School of System Informatics, Kobe University, Hyogo 657-8501, Japan.
			(email:{\tt  
				wakaiki@ruby.kobe-u.ac.jp).}}
		\thanks{
			This work was supported in part by JSPS KAKENHI Grant Number
			JP20K14362.}%
	}
	
	\maketitle

\begin{abstract}
We study the self-triggered stabilization of discrete-time 
linear systems with quantized state measurements.
In the networked control system we consider,
sensors may be spatially distributed and be connected to
a self-triggering mechanism through finite data-rate channels.
Each sensor independently encodes its measurements and
sends them to the self-triggering mechanism.
The self-triggering mechanism integrates 
quantized measurement data and then computes sampling times.
Assuming that the closed-loop system is stable in the absence 
of quantization and self-triggered sampling, we propose
a joint design method of an encoding scheme and 
a self-triggering mechanism  for stabilization.
To deal with data inaccuracy due to quantization,
the proposed self-triggering mechanism uses
not only quantized data but also an upper bound
of quantization errors, which is shared with a decoder. 
\end{abstract}

\begin{IEEEkeywords}
Networked control systems, quantized control, self-triggered control.
\end{IEEEkeywords}

\section{Introduction}
The subject of this note is self-triggered control
with quantized state measurements. Quantized control and self-triggered control 
have been extensively studied in the past few decades.
In both research areas, 
many methods have been developed for control
with limited information about plant measurements.
However, a synergy between quantized control and
self-triggered control has not been studied sufficiently.
It is our aim to combine these two research areas. 
In particular, we construct a self-triggering mechanism
	that determines sampling times for stabilization 
	from quantized measurements of
	possibly spatially distributed sensors.

Signal quantization is unavoidable for data transmission over
digital communication channels.
Coarse quantization may make feedback systems unstable.
Moreover, asymptotic convergence to equilibrium points cannot
be achieved by static finite-level quantizers in general.
Time-varying quantizers for stabilization with finite data rates
have been developed in \cite{Brockett2000,Liberzon2003}.
This class of time-varying quantizers
has been introduced for linear time-invariant systems and 
then has been extended to more general classes of systems such as 
nonlinear systems \cite{Liberzon2003Automatica, Liberzon2005}, 
switched linear systems \cite{Liberzon2014, Wakaiki2017TAC}, and 
systems under DoS attacks \cite{Wakaiki2020DoS}. 
Instability due to quantization errors
raises also a theoretical question of how coarse quantization is allowed
without compromising the closed-loop stability.
From this motivation,
data-rate limitation for stabilization has been extensively 
investigated; see the surveys \cite{Nair2007, Ishii2012}.

To reduce resource utilization,
techniques for aperiodic data transmission
have attracted considerable attention.
Event-triggered control \cite{Aarzen1999, Astrom2002} 
	and self-triggered control \cite{Velasco2003}
are the two major approaches of the aperiodic transmission techniques.
In both event-triggered control systems and 
self-triggered control systems,
the transmission of information occurs only when needed.
In event-triggered control systems,
triggering conditions are based on current measurements and
are monitored continuously or periodically.
Instead of such frequent monitoring, self-triggering mechanisms compute
the next transmission time when they receive measurements.
The advantage of self-triggered control systems
is that the sensors can be deactivated between transmission times.
Various triggering mechanisms, together with stability analysis, have been proposed; see, e.g.,
\cite{Tabuada2007,Heemels2008,Heemels2013} for the event-triggered case and
\cite{Wang2009, Anta2010,Mazo2010} for the self-triggered case.
Moreover, joint design methods of feedback control laws and 
triggering mechanisms have been developed 
in \cite{Ghodrat2020, Xue2021, Lu2020, Wan2021} and the references therein.

Quantized event-triggered  control has become an
active research topic in recent years; see, e.g., 
\cite{Garcia2013, Li2016ETC, Tanwani2016ETC ,Du2017,Liu2018ETC, Abdelrahim2019,Wang2021} 
and the references therein. However,
there has been relatively little work on 
quantized self-triggered  control.
A consensus protocol
with a quantized self-triggered communication policy
has been proposed for multi-agent systems 
in \cite{Persis2013,Matsume2021}, but
these systems differ significantly from the models we study. 
Sum-of-absolute-values optimization 
has been employed
for self-triggered control with discrete-valued inputs
in \cite{Ikeda2016}.
In \cite{Zhou2018STC}, self-triggered and event-triggered control with input
and output quantization has been studied. 
However, 
the self-triggering mechanisms proposed in \cite{Ikeda2016,Zhou2018STC}
need
the non-quantized measurements, which
would remove difficulties present in the computation of sampling times.
Many technical tools are commonly used for 
quantized control and self-triggered control.
This is because
analyzing implementation-induced errors plays a crucial role
in both research areas.
Hence coupling these two research areas is quite natural.

 \begin{figure}[!b]
	\centering
	\includegraphics[width = 6.5cm]{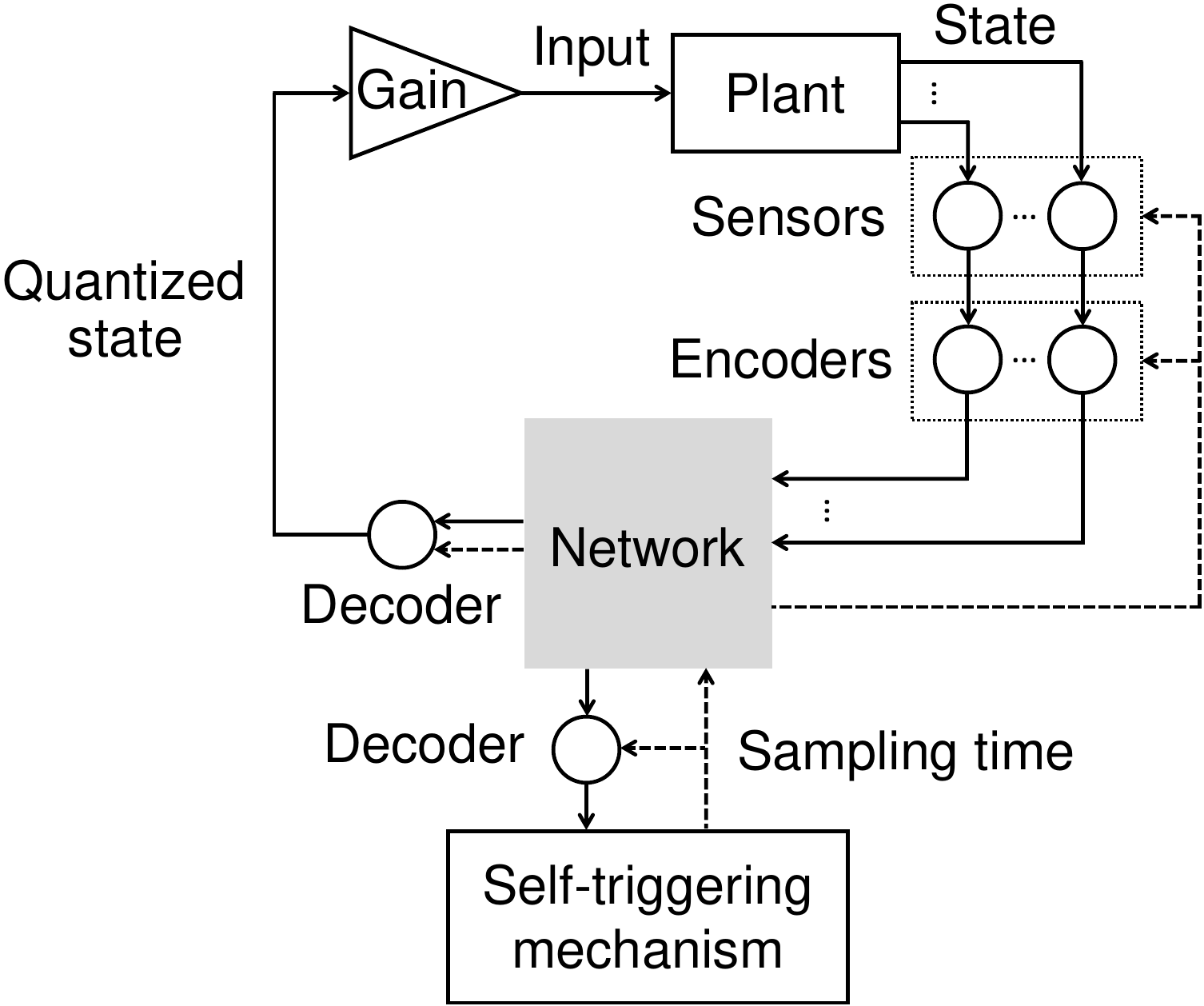}
	\caption{Networked control system. Since the encoding and decoding of
		sampling times is simple and  not essential in our discrete-time
		setting, we omit
		it in the figure.}
	\label{fig:closed_loop}
\end{figure}

In this note,
we consider the networked control system shown in Fig.~\ref{fig:closed_loop} and
assume 
that the  system
is stable when no quantization or self-triggering sampling is performed.
Our main contribution  is to develop a joint design
method of an encoding scheme and a self-triggering mechanism for stabilization.
The proposed encoding and self-triggering strategy has the
following advantageous
features:
\begin{itemize}
	\item
	The proposed self-triggering mechanism 
determines sampling times from the quantized state, unlike
the self-triggering mechanisms developed in \cite{Ikeda2016,Zhou2018STC} that 
use the original (non-quantized) state. Due to this property, we do not need to
install the self-triggering mechanism at the sensors. Therefore, the proposed encoding and self-triggering strategy 
is applicable to the scenario in which the sensors 
do not have computational resources enough to determine sampling times
by self-triggering mechanisms; see also \cite{Akashi2018}
for the computational issue of self-triggered control.

\item
In the proposed encoding scheme, an individual 
sensor encodes its measurement data 
without information from other sensors. In contrast, the existing 
scheme proposed in \cite{Zhou2018STC} has to collect 
measurement data from all sensors in one place. This issue 
does not arise in the previous study \cite{Ikeda2016} because
it considers only input quantization for single-input systems.
The distributed architecture 
allows the proposed encoding scheme to be applied to
systems with spatially distributed sensors.
\end{itemize}

	In contrast with the distributed architecture of the encoding scheme,
	the self-triggering mechanism 
	works in a centralized way, i.e., it integrates measurement data
	sent by all sensors in order to
	compute sampling times for stabilization.
	In this aspect,
the use of quantized measurements in the self-triggering mechanism
is also important when sensors are spatially distributed.
In fact, even when the self-triggering mechanism
is colocated with one sensor, it needs to 
receive 
measurement data from other distant sensors, which
is  done through digital channels in most cases.

The quantized self-triggered stabilization problem we study has 
two difficulties.
First, sampling times are computed 
only from inaccurate information on the plant state.
A key insight for solving this issue is that 
the self-triggering mechanism can share an upper bound of
quantization errors with the decoder.
To compensate for the inaccuracy of information on the state,
the proposed triggering mechanism exploits
not only the quantized state but also the upper bound of
quantization errors.
The second difficulty is that 
the self-triggered sampling makes the encoding and
decoding scheme aperiodic.
To deal with this aperiodicity,
we introduce in the analysis a new norm 
with respect to which the closed-loop matrix is 
a strict contraction.
The contraction property of the norm 
enables us to develop a simple update rule of the
encoding and decoding scheme, which
requires less computational resources in the encoders.

\color{black}
The remainder of this note is organized as follows.
In Section~\ref{sec:NCS}, the networked control system we consider
is introduced.
In Section~\ref{sec:stabilization}, we propose a joint design
method of an encoding scheme and a self-triggering mechanism for
stabilization.
We illustrate the proposed method with a numerical example in Section~\ref{sec:example}
and give concluding remarks in Section~\ref{sec:conclusion}.

\paragraph*{Notation}
The set of non-negative integers and
the set of non-negative real numbers are denoted by $\mathbb{N}_0$ and
$\mathbb{R}_+$, respectively.
Let $A^{\top}$ be
the transpose of a matrix $A \in \mathbb{R}^{m\times n}$.
Let $I_n$ denote the identity matrix of order $n$.
For a vector $v \in \mathbb{R}^{ n}$ with $i$th element $v_i$, 
its maximum norm is $\|v\|_{\infty} := 
\max\{|v_i |,\dots, |v_n|\}$.
The corresponding induced norm of a matrix 
$A \in \mathbb{R}^{ m\times n}$ 
with $(i,j)$th element $A_{ij}$
is given by
$\|A\|_{\infty} = \max\{\sum_{j=1}^n |A_{ij}| : 1 \leq i \leq m\} $.
We denote by $\varrho(P)$ 
the spectral radius of $P \in \mathbb{R}^{n\times n}$.
For a matrix sequence 
$\{A_k\}_{k \in \mathbb{N}_0} \subset \mathbb{R}^{ m\times n}$,
the empty sum $\sum_{k=0}^{-1} A_k$ is set to $0$.
 
\section{Networked control system}
\label{sec:NCS}
In this section, the control system we consider and 
	a basic encoding and decoding scheme are introduced. We also present
the  structure of the proposed self-triggering mechanism.
\subsection{Plant and controller}
Consider the discrete-time linear time-invariant system
\begin{equation}
\label{eq:system}
\left\{
\begin{alignedat}{4}
x(k+1) &= Ax(k) + Bu(k), & \qquad  &k \in \mathbb{N}_0,\\
u(k) &= Kq_\ell,& &k_\ell \leq k < k_{\ell+1},
\end{alignedat}
\right.
\end{equation}
where $x(k) \in \mathbb{R}^n$ and $u(k) \in \mathbb{R}^m$
are the state and the input of the plant at time $k \in \mathbb{N}_0$, 
respectively.
The time sequence $\{k_\ell\}_{\ell \in \mathbb{N}_0}$ with $k_0 :=0$ is
computed by a certain self-triggering mechanism, and 
$q_\ell$ is the quantized value of $x(k_\ell)$.

Define the closed-loop matrix $A_{\text{cl}}$ by
$A_{\text{cl}} := A+BK$.
We assume that the closed-loop system is stable
in the situation where the state $x(k)$ is transmitted 
without quantization at all times $k \in \mathbb{N}_0$.
\begin{assumption}
	\label{assump:stability}
	The feedback gain $K$ is chosen so that 
	the closed-loop matrix $A_{\text{cl}}$ is Schur stable, that is,
	there exist constants $\Gamma \geq 1$ and $\gamma \in (0,1)$ such that 
	\begin{equation}
	\label{eq:Acl_bound}
	\big\|A_{\text{cl}}^k
	\big\|_{\infty} \leq \Gamma \gamma^k\qquad \forall k \in \mathbb{N}_0.
	\end{equation}
	\vspace{-9pt}
\end{assumption}

We also place an assumption that a bound of 
the initial state $x(0)$ is known. 
One can obtain an initial state bound from the
standard zooming-out procedure developed in \cite{Liberzon2003},
where quantized signals are assumed to be transmitted at every time.
\begin{assumption}
	\label{assump:initial_bound}
	A constant $E_0>0$ satisfying $\|x(0)\|_{\infty} \leq E_0$ is known.
\end{assumption}

In this note, we study the following notion of
the closed-loop stability.
\begin{definition}
	The discrete-time system \eqref{eq:system} achieves 
	{\em exponential
		convergence} under 
	Assumption~\ref{assump:initial_bound} 
	if there exist constants $\Omega \geq 1$ and $\omega \in (0,1)$,
	independent of $E_0$,
	such that 
	\begin{equation}
	\label{eq:exp_conv}
	\|x(k)\|_{\infty} \leq \Omega E_0 \omega^k \qquad \forall k \in \mathbb{N}_0
	\end{equation}
	for every initial state $x(0) \in \mathbb{R}^n$ satisfying 
	$\|x(0)\|_{\infty} \leq E_0$.
\end{definition}

\begin{remark}
	Consider the continuous-time linear time-invariant system
	\begin{equation}
	\label{eq:conk_system}
	\dot x_{\text{c}}(t) = A_{\text{c}}x_{\text{c}}(t) + B_{\text{c}}u_{\text{c}}(t),  \quad  t \geq 0,
	\end{equation}
	where $x_{\text{c}}(t) \in \mathbb{R}^n$ and $u_{\text{c}}(t) \in \mathbb{R}^m$
	are the state and the input of the plant at time $t \geq 0$, 
	respectively. A standard self-triggered mechanism is given by
	\[
	t_{\ell+1} := t_\ell + \inf \{\tau > 0
	: f(x(t_\ell),\tau) >0  \},\quad \ell \in \mathbb{N}_0
	\]
	for some function $f:\mathbb{R}^n\times \mathbb{R}_+ \to \mathbb{R}$.
	However, this mechanism has two implementation issues.
	First, the next sampling time $t_{\ell+1}$ (or 
	the inter-sampling time $t_{\ell+1} - t_\ell$) needs to be quantized when
	it is sent to the sensors over finite data-rate channels.
	Second, the triggering mechanism 
	has to check the condition $f(x(k_\ell),\tau) >0$
	continuously with respect to $\tau > 0$.
	An easy way to circumvent these issues is to 
	place a time-triggering condition $\{t_\ell\}_{\ell \in \mathbb{N}_0} \subset \{ 
	\ell h
	\}_{\ell \in \mathbb{N}_0}$ for some $h>0$ 
	as in 
	the self-triggering mechanism proposed in \cite{Mazo2010} and 
	the periodic event-triggering mechanism (see, e.g., 
	\cite{Heemels2013}).
	When the continuous-time system \eqref{eq:conk_system} 
	is discretized with period $h$ under this 
	time-triggering condition, the resulting
	discrete-time system is in the form \eqref{eq:system}, where
	the matrices $A$ and $B$ are given by
	\[
	A = e^{A_{\text{c}} h},\quad B = \int^h_0 e^{A_{\text{c}} \tau} d\tau B_{\text{c}}
	\]
	and the state $x(k)$ and the input $u(k)$ are 
	$x(k) = x_{\text{c}}(k h)$ and $u(k) = u_{\text{c}}(kh)$ for $k \in \mathbb{N}_0$.
\end{remark}

\subsection{Basic encoding and decoding scheme}
\label{sec:EDS}
Let $\ell \in \mathbb{N}_0$, and
assume that 
we have obtained $E_\ell > 0$ satisfying
$\|x(k_\ell)\|_{\infty} \leq E_\ell$ at
the $\ell$th sampling time $k=k_\ell$.
In the next section,
we will explain how to obtain such a bound  $E_\ell$; see
\eqref{eq:E_def} and  Lemma~\ref{lem:state_bound} below for 
details.

Let 
$\eta \in \mathbb{N}$ be the number of sensors, and let
$n_1,\dots,n_{\eta} \in \mathbb{N}$ satisfy
$n =  n_1 + \cdots + n_\eta$.
We partition the state $x(k_\ell)$  into
\[
x(k_\ell) = \begin{bmatrix}
x^{\langle 1 \rangle}(k_\ell) \\ \vdots \\  x^{\langle \eta \rangle}(k_\ell)
\end{bmatrix},
\]
where $x^{\langle i \rangle}(k_\ell) \in \mathbb{R}^{n_i}$ 
is measured by the $i$th sensor
for $i =1,\dots,\eta$.
By assumption, $x^{\langle i \rangle}(k_\ell)$ satisfies $\|x^{\langle i \rangle}(k_\ell)\|_{\infty} \leq E_\ell$.
Let $N \in \mathbb{N}$ be the number of
quantization levels per dimension.
The $i$th encoder divides the hypercube
\[
\Big\{ 
x^{\langle i \rangle} \in \mathbb{R}^{n_i} : 
\big\|x^{\langle i \rangle} \big\|_{\infty} \leq E_\ell
\Big\}
\]
into $N^{n_i}$ equal hypercubes. Indices $\{1,\dots, N^{n_i}\}$
are assigned to divided hypercubes by a certain one-to-one
mapping. The $i$th encoder 
sends the index of the divided hypercube 
containing $x^{\langle i \rangle}(k_\ell)$ to 
the decoders at 
the self-triggering mechanism and the feedback gain.
If $x^{\langle i \rangle}(k_\ell)$ lies on the boundary of 
several hypercubes, then either one of these hypercubes can be chosen.
The decoders 
calculate the value of the center of the hybercube corresponding
the received index, and 
the quantized value $q_\ell^{\langle i \rangle}$ 
of $x^{\langle i \rangle}(k_\ell)$ is set to this value.
By construction, we obtain
\begin{equation}
\label{eq:i_qe}
\big\|q_\ell^{\langle i \rangle} - x^{\langle i \rangle}(k_\ell)\big\|_{\infty} \leq 
\frac{E_\ell}{N}.
\end{equation}
Define 
\[
q_\ell := \begin{bmatrix}
q_\ell^{\langle 1 \rangle} \\ \vdots \\ q_\ell^{\langle \eta \rangle}
\end{bmatrix} \in \mathbb{R}^n.
\]
Then \eqref{eq:i_qe} yields
\begin{equation}
\label{eq:quantization_error}
\| q_\ell - x(k_\ell) \|_{\infty} \leq 
\frac{E_\ell}{N}.
\end{equation}

\subsection{Structure of self-triggering mechanism}
\label{sec:STM}
The sampling times $\{k_\ell\}_{\ell \in \mathbb{N}_0}$
is generated by a self-triggering mechanism of the form
\begin{equation}
\label{eq:STM}
\left\{
\begin{aligned}
k_{\ell+1} &:= k_\ell + \min\{\tau_{\max},~\tau_\ell \},\quad 
k_0 := 0, \\
\tau_{\ell} &:= \min \{ \tau \in \mathbb{N}: g(q_\ell , E_\ell , \tau) > \sigma E_\ell \},
\quad \ell \in \mathbb{N}_0,
\end{aligned}
\right.
\end{equation}
where $\sigma >0$ is a threshold parameter, $\tau_{\max} \in \mathbb{N}$ is 
an upper bound of inter-sampling times $k_{\ell+1} - k_\ell$, that is,
$k_{\ell+1} - k_\ell \leq \tau_{\max}$ for every $\ell \in \mathbb{N}_0$, and
$g :\mathbb{R}^{n} \times \mathbb{R}_+ \times \mathbb{N}_0 \to \mathbb{R}_+$
is a certain function. The details of $g$ will be given in the next section; see
\eqref{eq:g_def} below.
The self-triggering mechanism \eqref{eq:STM} determines
the next sampling time $k_{\ell+1}$ from the quantized state $q_{\ell}$ 
and the state bound $E_{\ell}$ without using the original state $x(k_{\ell})$.
Therefore, it does not need to be installed at the sensors.
Note that the self-triggering mechanism 
knows from the state bound $E_{\ell}$ 
that the quantization error does not exceed 
$E_{\ell}/N$ by \eqref{eq:quantization_error}.

The inter-sampling time
$\min\{\tau_{\max},~\!\tau_\ell \}$ is transmitted to
the sensors, and the sensors measure the state 
at $k_{\ell+1} = k_\ell + \min\{\tau_{\max},~\!\tau_\ell \}$.
Setting the upper bound $\tau_{\max}$ allows
the self-triggering mechanism to inform the sensors about 
the next sampling instant with a finite data rate.
Since inter-sampling times
can be transmitted by a  simple encoding and decoding scheme,
we omit the details.

	In contrast to the distributed encoding scheme described in  Section~\ref{sec:EDS}, the sampling times 
	$\{k_{\ell}\}_{\ell \in \mathbb{N}_0}$
	are computed in a centralized manner, that is,
	the quantized data from all the sensors are collected 
	in the self-triggering mechanism \eqref{eq:STM} 
	for the computation of $\{k_{\ell}\}_{\ell \in \mathbb{N}_0}$. 
	Individual sensors cannot determine the next sampling time
	by themselves due to the lack of information on other measurement data
	(and also of computational resources in some cases).
	To compute sampling times for stabilization,
	the centralized self-triggering mechanism \eqref{eq:STM}
	integrates measurement data.

\section{Quantized self-triggered stabilization}
\label{sec:stabilization}
The encoding and self-triggering strategy presented in
Sections~\ref{sec:EDS} and \ref{sec:STM} is completely determined
if the following two components are given:
\begin{itemize}
	\item the sequence $\{E_{\ell}\}_{\ell \in \mathbb{N}_0}$ of 
	state bounds for the encoding and decoding scheme;
	\item the function $g$ in  the self-triggering mechanism \eqref{eq:STM}.
\end{itemize}
In this section, we first construct the function $g$, after analyzing 
errors due to quantization and self-triggered sampling.
Next, we design the sequence 
$\{E_{\ell}\}_{\ell \in \mathbb{N}_0}$ of state bounds
under sampling times computed by the self-triggering mechanism
\eqref{eq:STM} with this function $g$.
After these preparations, we provide a sufficient condition for
the quantized self-triggered control system to achieve exponential convergence.
Finally, we summarize 
the proposed 
joint design of an encoding scheme and a self-triggering mechanism
for stabilization.

\subsection{Error analysis for self-triggered sampling}
We construct the function $g$ in  the 
self-triggering mechanism \eqref{eq:STM}
so that the input error $\|Kq_\ell - Kx(k)\|_{\infty}$
satisfies
\begin{equation}
\label{eq:input_error}
\|Kq_\ell - Kx(k)\|_{\infty} \leq \sigma E_{\ell}
\end{equation}
for all $k_\ell + 1 \leq k < k_{\ell+1}$ and $\ell \in \mathbb{N}_0$.
To this end,
we first obtain an upper bound of the input error.
\begin{lemma}
	\label{lem:inpuk_error}
	{\em
	Let 
	$\ell \in \mathbb{N}_0$ and suppose that 
	the system \eqref{eq:system} with the encoding 
	and decoding scheme described in Section~\ref{sec:EDS}
	satisfies
	$\|x(k_\ell)\|_{\infty} \leq  E_\ell$ for some $E_{\ell} >0$.
	Then the quantized state $q_\ell$ satisfies
	\begin{align}
	\|Kq_\ell - Kx(k)\|_{\infty}
	&\leq 
	\left\|K \hspace{-1.3pt}\left(I_n- A^{k-k_\ell} - \sum_{\tau=0}^{k-k_\ell-1} A^\tau BK \right)\hspace{-1.3pt}q_\ell\right\|_{\infty} \notag \\
	&\qquad + \big\|KA^{k-k_\ell}\big\|_{\infty} \frac{E_\ell}{N} 
	\label{eq:inpuk_error_bound}
	\end{align}
	for all $k_\ell \leq k < k_{\ell+1}$.
}
\end{lemma}
\begin{proof}
	Let 
	$\ell \in \mathbb{N}_0$ and $k_{\ell} \leq k < k_{\ell+1}$.
	Since
	\begin{align*}
	x(k)&=A^{k-k_\ell} x(k_\ell) + 
	\sum_{\tau=0}^{k-k_\ell-1} A^{\tau} BKq_\ell \\
	&= \left( A^{k-k_\ell} + 
	\sum_{\tau=0}^{k-k_\ell-1} A^{\tau} BK\right)q_\ell - 
	A^{k-k_\ell} \big(q_{\ell} - x(k_{\ell}) \big),
	\end{align*}
	it follows that 
	\begin{align*}
	q_\ell - x(k) &= 
	\left( I_n- A^{k-k_\ell} - \sum_{\tau=0}^{k-k_\ell-1} A^{\tau} BK \right) q_\ell \\
	&\qquad + 
	A^{k-k_\ell} \big( q_{\ell} - x(k_{\ell})\big).
	\end{align*}
	Thus, 
	the inequality
	\eqref{eq:inpuk_error_bound} follows from \eqref{eq:quantization_error}.
\end{proof}

We define 
the function $g$
in the self-triggering mechanism \eqref{eq:STM}
by 
\begin{equation}
\label{eq:g_def}
g(q,E,\tau) := 
\left\|K\hspace{-1.3pt} \left(I_n- A^\tau - \sum_{p=0}^{\tau-1} A^pBK \right)\hspace{-1.3pt}q\right\|_{\infty} +
\big\|KA^{\tau}\big\|_{\infty} \frac{E}{N}
\end{equation}
for $q \in \mathbb{R}^n$, $E \geq 0$, and $\tau \in \mathbb{N}_0$.
Lemma~\ref{lem:inpuk_error} shows that 
if $E_{\ell} >0$ satisfies $\|x(k_\ell)\|_{\infty} \leq  E_\ell$, then
\[
\|Kq_\ell - Kx(k)\|_{\infty} \leq g(q_{\ell},E_{\ell},k-k_{\ell})
\] 
for all $k_\ell \leq k < k_{\ell+1}$.
Combining this and the triggering condition given in \eqref{eq:STM},
we obtain the desired inequality \eqref{eq:input_error} for all $k_\ell + 1 \leq k < k_{\ell+1}$ and $\ell \in \mathbb{N}_0$.

\subsection{Generating state bounds for encoding-decoding scheme}
To complete the design of the encoding and decoding scheme
described in Section~\ref{sec:EDS},
we next construct a sequence
$\{E_\ell\}_{\ell \in \mathbb{N}_0}$ 
satisfying $\|x(k_{\ell})\|_{\infty} \leq E_{\ell}$ for all $\ell \in \mathbb{N}_0$.
Note that the sampling times $\{k_{\ell}\}_{k \in \mathbb{N}_0}$ are 
computed by the self-triggering mechanism \eqref{eq:STM} with
 the function $g$ in \eqref{eq:g_def}.

Using the
constants $\Gamma\geq 1$ and $\gamma \in (0,1)$ satisfying \eqref{eq:Acl_bound},
we define 
$\{E_\ell\}_{\ell \in \mathbb{N}_0}$
by
\begin{equation}
\left\{
\begin{alignedat}{4}
\label{eq:E_def}
E_{\ell} &:=  \widetilde E_\ell,\quad \ell \in \mathbb{N}, \\
\widetilde E_0 &:= \Gamma E_0, \\
\widetilde E_{\ell+1} &:= \big(\gamma^{k_{\ell+1}-k_\ell} (1- \delta \sigma
) + \delta \sigma \big)
\widetilde E_\ell,\quad \ell \in \mathbb{N}_0,
\end{alignedat}
\right.
\end{equation}
where 
\[
\delta := \frac{\Gamma \|B\|_{\infty} }{1-\gamma}.
\]
In the periodic sampling case such as \cite{Liberzon2003,Liberzon2005,Liberzon2014,Wakaiki2020DoS},
the decay rate of $\{E_{\ell}\}_{\ell \in \mathbb{N}_0}$ depends on 
the number $N$ of quantization levels.
However, 
the update rule \eqref{eq:E_def} 
uses only the threshold parameter $\sigma$ and 
the inter-sampling time $k_{\ell+1} - k_{\ell}$.
The self-triggering mechanism exploits the advantage of
small quantization errors for reducing the number of 
data transmissions.
Consequently, the number $N$ of quantization levels  does not directly 
affect
the decay rate of $\{E_{\ell}\}_{\ell \in \mathbb{N}_0}$.

The following result provides a simple condition for the hypercube $\{x \in \mathbb{R}^n:
\|x\|_{\infty} \leq E_{\ell} \}$ to contain the state $x(k_{\ell})$.
\begin{lemma}
	\label{lem:state_bound}
	{\em
	Suppose that  Assumption~\ref{assump:initial_bound} hold.
	Let 
	the time sequence $\{k_\ell \}_{\ell \in \mathbb{N}_0}$ be
	as in \eqref{eq:STM}, where the function $g$ is defined by
	\eqref{eq:g_def}.
	Take
	a number $N \in \mathbb{N}$ of 
	quantization levels  and
	a thereshold parameter $\sigma>0$ such that
	\begin{equation}
	\label{eq:sigN_cond}
	\frac{\|K\|_{\infty}}{N} \leq \sigma.
	\end{equation}
	Then the state $x$ of the system \eqref{eq:system} satisfies
	\begin{equation}
	\label{eq:x_bound}
	\|x(k_\ell)\|_{\infty} \leq E_\ell\qquad \forall \ell \in \mathbb{N}_0,
	\end{equation}
	where the sequence $\{E_\ell\}_{\ell \in \mathbb{N}_0}$ 
	is defined by \eqref{eq:E_def}.
}
\end{lemma}

To prove this lemma,
we use a norm $\|\cdot\|_{\text{cl}}$ with respect to which 
the closed-loop matrix $A_{\text{cl}}$ is a strict contraction, i.e.,
$\|A_{\text{cl}} \xi\|_{\text{cl}} < \|\xi\|_{\text{cl}} $ 
for all nonzero $\xi \in \mathbb{R}^n$, under Assumption~\ref{assump:stability}.
Such a norm was constructed for
infinite-dimensional systems in \cite[Lemma~II.1.5]{Eisner2010} and \cite{Wakaiki2018_EVC} without detailed proof.
We state the finite-dimensional version in the following lemma
and include the proof in Appendix for completeness.
\begin{lemma}
	\label{lem:contraction}
	{\em 
	Let $F \in \mathbb{R}^{n \times n}$, 
	$\Gamma\geq 1$, and $\gamma >0$ satisfy
	\[
	\big\| F^k\big\|_{\infty} \leq \Gamma \gamma^k\qquad \forall k \in \mathbb{N}_0.
	\]
	Then
	the function
	\begin{align*}
	\vertiii {\cdot }&: \mathbb{R}^n \to \mathbb{R} \\
	&: \xi \mapsto 	\vertiii {\xi } := \sup_{k \in \mathbb{N}_0}
	\big\|\gamma^{-k} F^k \xi\big\|_{\infty}
	\end{align*}
	is a norm on $\mathbb{R}^n$.
	Moreover, the norm $\vertiii {\cdot }$ satisfies
	\begin{equation}
	\label{eq:prop1}
	\|\xi\|_{\infty}\leq \vertiii {\xi } \leq \Gamma \|\xi\|_{\infty}	\qquad \forall \xi \in \mathbb{R}^n
	\end{equation}
	and
	\begin{equation}
	\label{eq:prop2}
	\bigvertiii{F^k \xi } \leq \gamma^k\vertiii {\xi } \qquad \forall \xi \in \mathbb{R}^n,~\forall k \in \mathbb{N}_0.
	\end{equation}
	\vspace{-9pt}
}
\end{lemma}

Under Assumption~\ref{assump:stability}, there exist constants
 $\Gamma \geq 1$ and 
$\gamma \in (0,1)$ such that $\|A_{\text{cl}}^k
\|_{\infty} \leq \Gamma \gamma^k$ for  all $k \in \mathbb{N}_0$.
Using the constants $\Gamma$ and 
$\gamma$, 
we define a new norm on
$\mathbb{R}^n$ by
\begin{equation}
\label{eq:cl_norm_def}
\|\xi\|_{\text{cl}} :=  \sup_{k \in \mathbb{N}_0}
\big\|\gamma^{-k} A_{\text{cl}}^k \xi\big\|_{\infty},\quad \xi \in \mathbb{R}^n.
\end{equation}
Lemma~\ref{lem:contraction} 
shows that the matrix $A_{\text{cl}}$
is a strict contraction with respect to the norm $\|\cdot \|_{\text{cl}}$.

\color{black}
\noindent\hspace{1em}{\itshape {\it Proof of Lemma~\ref{lem:state_bound}:} }
	By Lemma~\ref{lem:contraction},
the norm $\|\cdot\|_{\text{cl}} $ defined as in \eqref{eq:cl_norm_def}
satisfies
\begin{equation}
\label{eq:Acl_norm_prop1}
\|\xi \|_{\infty}  \leq \|\xi \|_{\text{cl}} \leq \Gamma \|\xi \|_{\infty} 
\qquad \forall \xi \in \mathbb{R}^n
\end{equation}
and
\begin{equation}
\label{eq:Acl_norm_prop2}
\big\|A_{\text{cl}}^k\xi \big\|_{\text{cl}} \leq \gamma^k\|\xi \|_{\text{cl}} 
\qquad \forall \xi  \in \mathbb{R}^n,~\forall k \in \mathbb{N}_0.
\end{equation}
By the property \eqref{eq:Acl_norm_prop1}, we obtain
the desired inequality \eqref{eq:x_bound} if
\begin{equation}
\label{eq:xR_bound}
\|x(k_{\ell})\|_{\text{cl}} \leq \widetilde E_\ell\qquad \forall 
\ell \in \mathbb{N}_0.
\end{equation}
Using the property~\eqref{eq:Acl_norm_prop1} and
Assumption~\ref{assump:initial_bound}, we obtain
\[
\|x(0)\|_{\text{cl}} \leq \Gamma \|x(0)\|_{\infty} \leq \Gamma E_0 =
\widetilde E_0.
\]
Hence
\eqref{eq:xR_bound} is
true for $\ell=0$.
We now proceed by induction and assume 
\eqref{eq:xR_bound} to be true for some 
$\ell \in \mathbb{N}_0$. Define 
$e(k) := q_\ell - x(k)$ for $k_\ell \leq k < k_{\ell+1}$, and set
$p_{\ell} := k_{\ell+1}-k_\ell$.
Then
\begin{align}
\label{eq:state_dynamics_kell1}
x(k_{\ell+1}) = A_{\text{cl}}^{p_{\ell} }  x(k_\ell) + 
\sum_{\tau=0}^{p_{\ell} -1} 
A_{\text{cl}}^{p_{\ell} -\tau - 1}BK e(k_\ell+\tau).
\end{align}
By \eqref{eq:sigN_cond}, 
\begin{equation}
\label{eq:initial_bound}
g(q_\ell , E_\ell , 0) = \|K\|_{\infty} \frac{E_\ell}{N} 
\leq \sigma E_\ell .
\end{equation}
Under the self-triggering mechanism~\eqref{eq:STM},
we obtain
\begin{equation}
\label{eq:g_bound}
g(q_\ell , E_\ell , \tau) \leq  \sigma E_\ell 
\end{equation}
for all $1 \leq \tau < p_{\ell} $. 
By definition, $E_\ell = \widetilde E_{\ell}$ for all $\ell \in \mathbb{N}$ and
$ E_\ell \leq \Gamma E_{\ell} =
\widetilde E_{\ell}$ for $\ell = 0$.
Since $\|x(k_{\ell})\|_{\infty} \leq E_{\ell}$
by assumption, 
Lemma~\ref{lem:inpuk_error} 
in the combination with the inequalities \eqref{eq:initial_bound} 
and \eqref{eq:g_bound} yields
\[
\|Ke(k_\ell+\tau)\|_{\infty} \leq g(q_\ell , E_\ell , \tau) \leq \sigma  E_\ell
\leq \sigma \widetilde E_{\ell}
\]
for all $0 \leq \tau < p_{\ell} $.
Applying the properties \eqref{eq:Acl_norm_prop1} and \eqref{eq:Acl_norm_prop2}
to \eqref{eq:state_dynamics_kell1},
we obtain
\begin{align}
&\|x(k_{\ell+1})\|_{\text{cl}} \notag \\
&\quad \leq 
\gamma^{p_{\ell} } \|x(k_\ell) \|_{\text{cl}}  +  \Gamma
\sum_{\tau=0}^{p_{\ell} -1} \gamma^{p_{\ell} -\tau-1} \|B\|_{\infty} \|Ke(k_\ell+\tau) \|_{\infty} 
\notag \\
&\quad \leq 
\gamma^{p_{\ell} } \widetilde 
E_\ell + \sigma  \Gamma \|B\|_{\infty} \sum_{\tau=0}^{p_{\ell} -1} \gamma^{\tau} 
\widetilde E_\ell \notag \\
&\quad \leq 
\big(\gamma^{p_{\ell} } (1- \delta \sigma) + \delta \sigma \big) \widetilde E_\ell = \widetilde E_{\ell+1}.
\label{eq:Rxktau_bound}
\end{align}
Thus, \eqref{eq:xR_bound} holds for $\ell +1$.
\hspace*{\fill} $\blacksquare$

\subsection{Sufficient condition for exponential convergence}
The following theorem gives a sufficient condition for
the closed-loop system
to achieve exponential convergence. 
\begin{theorem}
	\label{thm:main}
	{\em
	Suppose that  Assumptions~\ref{assump:stability} and
	\ref{assump:initial_bound} hold.
	Construct the components $g$ 
	and $\{E_\ell\}_{\ell \in \mathbb{N}_0}$ of the encoding and 
	self-triggering strategy
	by \eqref{eq:g_def}  and \eqref{eq:E_def}, respectively.
	If the number $N \in \mathbb{N}$ of 
	quantization levels  and
	the thereshold parameter $\sigma>0$ satisfy
	\begin{equation}
	\label{eq:sigma_cond}
	\frac{\|K\|_{\infty}}{N} \leq 
	\sigma < \frac{1}{\delta} = \frac{1-\gamma }{\Gamma \|B\|_{\infty}},
	\end{equation}
	then the system \eqref{eq:system} 
	with the encoding and 
	self-triggering strategy described in Sections~\ref{sec:EDS} and \ref{sec:STM}
	 achieves 
	exponential	convergence. Moreover, the constant $\omega$ given by
	\begin{equation}
	\label{eq:omega_def}
	\omega := 
	\big(\gamma^{\tau_{\max}} 
	(1- \delta \sigma ) +\delta \sigma \big)^{1/\tau_{\max}}
	\end{equation}
	satisfies
	\eqref{eq:exp_conv} for some $\Omega \geq 1$.
}
\end{theorem}

Before proceeding to the proof of this theorem, 
we  provide some remarks on the obtained sufficient 
condition \eqref{eq:sigma_cond}.
The condition $\|K\|_{\infty} /N \leq \sigma$ 
is used to avoid $g(q_{\ell},E_{\ell},0) >
\sigma E_{\ell}$ for $\ell \in \mathbb{N}_0$ as shown in the proof of Lemma~\ref{lem:state_bound}; see \eqref{eq:initial_bound}. 
Without this condition, 
the input error due to quantization may be larger than
the threshold $\sigma E_{\ell}$ even at the sampling time $k=k_{\ell}$.
On the other hand, the condition $\sigma < 1/\delta$ is used to guarantee
exponential convergence. In fact, we see from
the definition \eqref{eq:E_def} of $\{E_\ell\}_{\ell \in \mathbb{N}_0}$ that
the condition $\sigma < 1/\delta$ is satisfied
if and only if $\{E_\ell\}_{\ell \in \mathbb{N}_0}$ is 
a decreasing sequence. Combing this fact with the
bound of the state obtained in \eqref{eq:x_bound},
we prove that the closed-loop system achieves
exponential convergence.

\noindent\hspace{1em}{\itshape {\it Proof of Theorem~\ref{thm:main}:} }
	Since  $\delta \sigma < 1$ by \eqref{eq:sigma_cond}, 
	it follows that
	\[
	\gamma^\tau (1-\delta \sigma) + \delta \sigma 
	< \gamma^0 (1-\delta \sigma) + \delta \sigma =
	1\qquad \forall \tau \in \mathbb{N}.
	\]
	Define 
	\[
	\omega := 
	\max_{1\leq \tau \leq \tau_{\max}}\big(\gamma^\tau (1-\delta \sigma) +\delta \sigma\big)^{1/\tau} < 1.
	\]
	By the definition \eqref{eq:E_def} of $\{E_\ell\}_{\ell \in \mathbb{N}_0}$,
	
	\[
	\widetilde E_{\ell+1} \leq  \widetilde E_\ell \omega^{k_{\ell+1}-k_\ell}
	\qquad \forall \ell \in \mathbb{N}_0.
	\]
	Therefore
	\begin{equation}
	\label{eq:wide_E_conv}
	E_{\ell} \leq
	\widetilde E_\ell \leq \widetilde E_0\omega^{k_\ell} = 
	\Gamma E_{0}\omega^{k_\ell} \qquad \forall \ell \in \mathbb{N}_0.
	\end{equation}
	For all $k_{\ell} \leq k < k_{\ell+1}$
	and  $\ell \in \mathbb{N}_0$, 
	\[
	x(k) = A^{k-k_\ell} x(k_\ell) + \sum_{\tau=0}^{k-k_{\ell}-1} A^{\tau}BK q_\ell.
	\]
	Lemma~\ref{lem:state_bound} gives
	\[
	\|x(k_{\ell})\|_{\infty} \leq E_{\ell}\qquad \forall \ell \in \mathbb{N}_0,
	\]
	and by construction, the quantized value $q_{\ell}$ also satisfies
	\[
	\|q_{\ell}\|_{\infty} \leq E_{\ell}\qquad \forall \ell \in \mathbb{N}_0.
	\] 
	Therefore,
	there exists $M\geq 1$ such that for all $k_{\ell} \leq k < k_{\ell+1}$
	and  $\ell \in \mathbb{N}_0$, 
	\[
	\|x(k)\|_{\infty} \leq M E_\ell.
	\]
	Combining this with \eqref{eq:wide_E_conv}, we obtain
	\begin{align*}
	\|x(k)\|_{\infty} 
	&\leq 
	( M \Gamma )
	E_{0} \omega^{k_\ell} 
	\\
	&\leq \big(\omega^{-\tau_{\max}}  M \Gamma \big)
	E_0 \omega^{k}
	\end{align*}
	for all $k_\ell \leq k < k_{\ell+1}$ and
	$\ell\in \mathbb{N}_0$. Thus,
	the system \eqref{eq:system} achieves 
	exponential	convergence.
	
	It remains to show that 
	\begin{equation}
	\label{eq:tau_max}
	\tau_{\max} = 
	\argmax_{1\leq \tau \leq \tau_{\max}}\big(\gamma^\tau (1-\delta \sigma) +\delta \sigma\big)^{1/\tau}.
	\end{equation}
	This fact was used in Theorem~5.8 of \cite{Wakaiki2018_EVC} without proof. 
	Here we give all details for the sake of completeness.
	
	We prove that 
	\[
	\phi(\tau) := \big(\gamma^\tau (1-\delta \sigma) +\delta \sigma\big)^{1/\tau}
	\]
	 is strictly increasing  on $[1,\infty)$.
	It suffices to show that 
	\[
	\Phi(\tau) := \log \phi (\tau) = \frac{\log( 
		\gamma^\tau (1-\delta \sigma) +\delta \sigma )}{\tau}
	\] 
	satisfies $\Phi'(\tau) >0$ for every 
	$\tau >0$.
	Define 
	\[
	\nu(\tau) := \gamma^\tau (1-\delta \sigma) +\delta \sigma.
	\]
	Since
	\[
	\Phi'(\tau) = \frac{\frac{\nu'(\tau)}{\nu(\tau)} \tau - \log \nu(\tau)}{\tau^2},
	\]
	it follows that $\Phi'(\tau) > 0$ if and only if
	\[
	\psi(\tau) := \frac{\nu'(\tau)}{\nu(\tau)} \tau - \log \nu(\tau) >0.
	\]
	We have that 
	\[
	\psi'(\tau) = \frac{\tau \big(\nu(\tau)\nu''(\tau) - \nu'(\tau)^2\big)}{\nu(\tau)^2}
	\]
	for all $\tau > 0$. Therefore,
	$\psi'(\tau) >0$ if and only if 
	\[\
	\nu(\tau)\nu''(\tau) - \nu'(\tau)^2 >0.
	\]
	Since
	\begin{align*}
	\nu'(\tau) &= \gamma^\tau(1-\delta \sigma)  \log \gamma \\
	\nu''(\tau) &= \gamma^\tau (1-\delta \sigma)  (\log \gamma)^2,
	\end{align*}
	it follows from $0 < \delta \sigma < 1$ that 
	\[
	\nu(\tau)\nu''(\tau) - \nu'(\tau)^2 = \gamma^{\tau} \delta \sigma (1-\delta \sigma)
	(\log \gamma)^2 >0
	\]
	for all $\tau >0$.
	Therefore, $\psi'(\tau) >0$. Since $\psi(0) = 0$, we obtain
	$\psi(\tau) >0$ and hence
	$\Phi'(\tau) >0$ for all $\tau >0$. Thus, \eqref{eq:tau_max} holds.
\hspace*{\fill} $\blacksquare$

\subsection{Design of encoding and self-triggering strategy}
Based on Theorem~\ref{thm:main}, we design 
an encoding and self-triggering strategy for
stabilization. Before doing so, we explain how to compute 
constants $\Gamma \geq 1$ and 
$\gamma \in (0,1)$ satisfying \eqref{eq:Acl_bound}.
First, we set a constant $\gamma \in (0,1)$. Next,
we numerically compute
a constant $\Gamma \geq 1$ 
corresponding to $\gamma$ as
\begin{equation}
\label{eq:Gamma_def}
\Gamma = \sup_{k \in \mathbb{N}_0} \big\|\gamma^{-k} A_\text{cl}^k \big\|_{\infty}.
\end{equation}
Let $\varrho(A_{\text{cl}})$ be the spectral radius of $A_{\text{cl}}$.
Every $\gamma >\varrho(A_{\text{cl}})$ satisfies \eqref{eq:Acl_bound}
for some $\Gamma \geq 1$.
On the other hand, if $\gamma <\varrho(A_{\text{cl}})$,
there does not exist a constant $\Gamma \geq 0$ such that \eqref{eq:Acl_bound}
holds. Note that 
a smaller $\gamma$ does not always allow a larger threshold parameter 
$\sigma$, because the constant $\Gamma$  given by \eqref{eq:Gamma_def} becomes larger as 
$\gamma$ decreases.

We summarize the proposed joint design of
an encoding scheme and a self-triggering mechanism
for exponential convergence.

\vspace{8pt} 
\noindent
\underline{\em{Encoding and self-triggering strategy}}
\begin{itemize}
\setlength{\leftskip}{18pt}
\item[Step~0.]
Take 
an upper bound $\tau_{\max}$ of inter-sampling times  and
a decay parameter $\gamma \in (\varrho(A_{\text{cl}}), 1)$.
Set $\Gamma := \sup_{k \in \mathbb{N}_0} \|\gamma^{-k} A_\text{cl}^k \|_{\infty}$, and choose a number $N \in \mathbb{N}$ of 
quantization levels  and
a threshold parameter $\sigma>0$
so that 
\begin{equation}
\label{eq:sigma_cond2}
\frac{\|K\|_{\infty}}{N} \leq  
\sigma < \frac{1-\gamma}{\Gamma  \|B\|_{\infty}}.
\end{equation}
\end{itemize}
At each sampling time $k_{\ell}$, 
the following information flow and computation occur.
\begin{itemize}
	\setlength{\leftskip}{18pt}
	\item[Step~1.] 
	The encoders generate
	the indices corresponding to the state $x(k_{\ell})$ 
	by  the scheme described in 
	Section~\ref{sec:EDS} and then transmit them
	to the self-triggering mechanism and the feedback gain.
	At both components, the indices are decoded to
	the quantized value $q_{\ell}$ of $x(k_{\ell})$.
\item[Step~2.]  The inter-sampling time 
$k_{\ell+1} - k_{\ell} \in \{1,\dots,\tau_{\max }\}$ is computed by 
the self-triggering mechanism~\eqref{eq:STM}, where the function 
$g$ is given by \eqref{eq:g_def}, and then  is sent to the sensors, 
the encoders, and the decoders.
\item[Step~3.] 
The encoders at the sensors 
calculate the next state bound $E_{\ell+1}$ by
the update rule \eqref{eq:E_def}.
The decoders at the self-triggering mechanism and the feedback gain 
also perform the same calculation.
\end{itemize}

\vspace{4pt}

We make some comments on the above strategy. First,
in Step~2, the inter-sample time $k_{\ell+1} - k_{\ell}$ is transmitted to
the encoders and the decoders. This is because they also utilize
inter-sampling times in Step~3 for the computation of the next state bound $E_{\ell+1}$. Second,
the distributed architecture described in Section~\ref{sec:EDS} and
the update rule \eqref{eq:E_def} of $\{E_{\ell}\}_{\ell \in \mathbb{N}_0}$
allow each sensor to encode its own measurements without using the measurements
of the other sensors. Hence,
the proposed strategy 
can be applied to the system whose
sensors  are spatially distributed.

We immediately see that the condition \eqref{eq:sigma_cond2}
holds for every sufficiently large number $N \in \mathbb{N}$ 
of quantization levels and
every sufficiently small threshold parameter $\sigma > 0$.
In other words, the
closed-loop system achieves exponential convergence under 
sufficiently fine quantization and fast self-triggered sampling.
Whether exponential convergence is achieved does not depend on
the upper bound $\tau_{\max}$ of inter-sampling times, but
the upper bound $\omega$ of the decay rate of the state given in \eqref{eq:omega_def}
becomes smaller as $\tau_{\max}$ increases.
	Note that $\omega$ depends on  $\sigma$
but not on $N$.
Fine quantization reduces the number of data transmissions in
the proposed encoding and self-triggering strategy, but
$\omega$ is determined only by the parameters of self-triggered sampling.

\begin{remark}
	Proposition~3.13 of \cite{Wakaiki2020DoS} provides
	another method to construct a norm with respect to which
	$A_{\text{cl}}$ is a strict contraction under Assumption~\ref{assump:stability}.
	In this method, an invertible matrix is a design parameter
	for the encoding and decoding scheme.
	Since a decay parameter 
	$\gamma$ is easier to tune than an invertible matrix,
	we here use Lemma~\ref{lem:contraction} for the construction
	of a new norm.
\end{remark}

\color{black}
\section{Numerical Example}
\label{sec:example}
We discretize the linearized model of the unstable batch reactor 
studied in \cite{Rosenbrock1972} with sampling period $h = 0.01$.
Then the matrices $A$ and $B$ in the state equation \eqref{eq:system} 
are given by
\begin{align*}
A &= \begin{bmatrix}
1.0142 &  -0.0018  &  0.0651  & -0.0546 \\
-0.0057  &  0.9582  & -0.0001  &  0.0067 \\
0.0103  &  0.0417  &  0.9363  &  0.0563 \\
0.0004  &  0.0417 &   0.0129  &  0.9797 
\end{bmatrix} \\
B &= 
10^{-2} \times
\begin{bmatrix}
 0.0005 & -0.1034 \\
5.5629  & 0.0002 \\
1.2511  &-3.0444 \\
1.2511 & -0.0205 
\end{bmatrix}.
\end{align*}
For this discretized system, we compute the
linear quadratic regulator whose state weighting matrix and
input weighting matrix are the diagonal 
matrices $I_4$ and $0.05\times I_2$, respectively.
The resulting feedback gain $K$ is given by
\[
K = 
\begin{bmatrix}
1.3565 & -3.3445 & -0.5501 & -3.8646 \\
5.8856  & -0.0462  &  4.5150 &  -2.4334
\end{bmatrix}.
\]
The closed-loop matrix $A_{\text{cl}} = A+BK$ is Schur stable,
and Assumption~\ref{assump:stability} is satisfied.
For the computation of time responses,
we take the initial state 
\[
x(0) = \begin{bmatrix}
-1 & -1 & - 1 & 1
\end{bmatrix}^{\top}.
\]
The initial state bound $E_0$ in Assumption~\ref{assump:initial_bound}
is set to $1.1$.

The spectral radius of the closed-loop matrix $A_{\text{cl}}$
is given by $\varrho(A_{\text{cl}}) = 0.9402$, and we set 
$\gamma = 1.01 \times \varrho(A_{\text{cl}})=
0.9496$. Then $\Gamma$ defined by \eqref{eq:Gamma_def} is 
$\Gamma = 2.6012$.
By Theorem~\ref{thm:main}, if the number $N$ of quantization levels  and 
the threshold parameter $\sigma$ satisfy
\begin{equation}
\label{eq:suf_cond}
\frac{12.8803}{N} \leq 
\sigma < 0.3482,
\end{equation}
then the closed-loop system achieves exponential convergence.
The parameters of the self-triggering mechanism
are given by $\sigma = 0.28$ and $\tau_{\max} = 20$,
and we consider two cases $N=61$ and $N=101$.
The condition \eqref{eq:suf_cond} is satisfied in 
both cases. In what follows, 
we compare the time responses between the cases 
 $N=61$ and $N=101$.
 
\begin{figure}[tb]
	\centering
	\includegraphics[width = 8cm]{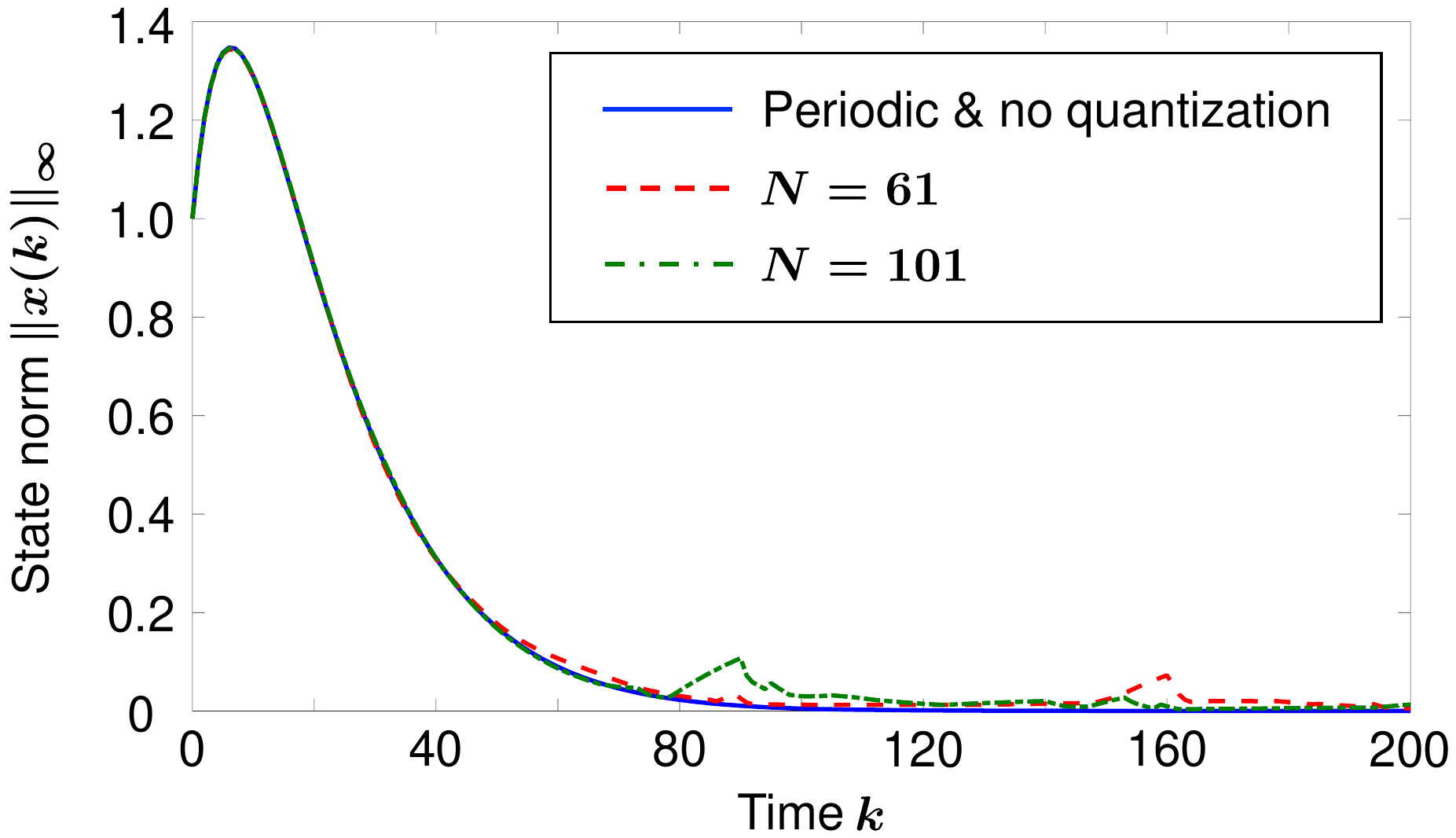}
	\caption{State norm $\|x(k)\|_{\infty}$.}
	\label{fig:state_norm}
\end{figure}
Fig.~\ref{fig:state_norm} shows the time responses of the state norm $\|x(k)\|_{\infty}$.
The blue solid line shows the ideal case where
the state $x(k)$ is transmitted at every $k \in \mathbb{N}_0$ 
without quantization.
The red dashed line and the green dotted line indicate
the cases $N=61$ and $N=101$, respectively.
We see from Fig.~\ref{fig:state_norm} that 
the state norm  converges to zero in both cases.
The convergence speeds have little difference
between
the cases $N=61$ and $N=101$, although 
quantization errors become smaller as $N$ increases.
This is because $N$ is related to
the number of data transmissions rather than  the convergence speed.

To see this, we plot the inter-sampling times $k_{\ell+1} - k_{\ell}$ 
for the cases 
$N=61$ and $N=101$ in Figs.~\ref{fig:ie_time_N61}
and \ref{fig:ie_time_N101}, respectively.
We see from these figures that 
the inter-sampling times in the case $N=101$ are
larger than those in the case $N=61$. 
In particular, the number of data transmissions for $k \geq 40$ 
is significantly reduced by the self-triggering mechanism 
in the case $N=101$.
The total numbers of data transmissions in the time-interval $[0,200]$ 
are 
$62$ for the case $N =61$ and $37$ for the case $N = 101$.
The amount of data per transmission in  the case $N = 101$
is $(101/61)^4  = 7.5156$ times larger than that in  the case $N=61$.
Hence the total amount of transmitted data
in the time-interval $[0,200]$ 
for the case $N=61$
is smaller than that for the case $N=101$. Note, however, that an important
benefit to be gained from fine quantization is that the sensors
can save energy and extend their lifetime, by reducing the number of 
sampling.

\begin{figure}[tb]
	\centering
	\subcaptionbox{Case $N=61$. \label{fig:ie_time_N61}}
	{\includegraphics[width = 8cm,clip]{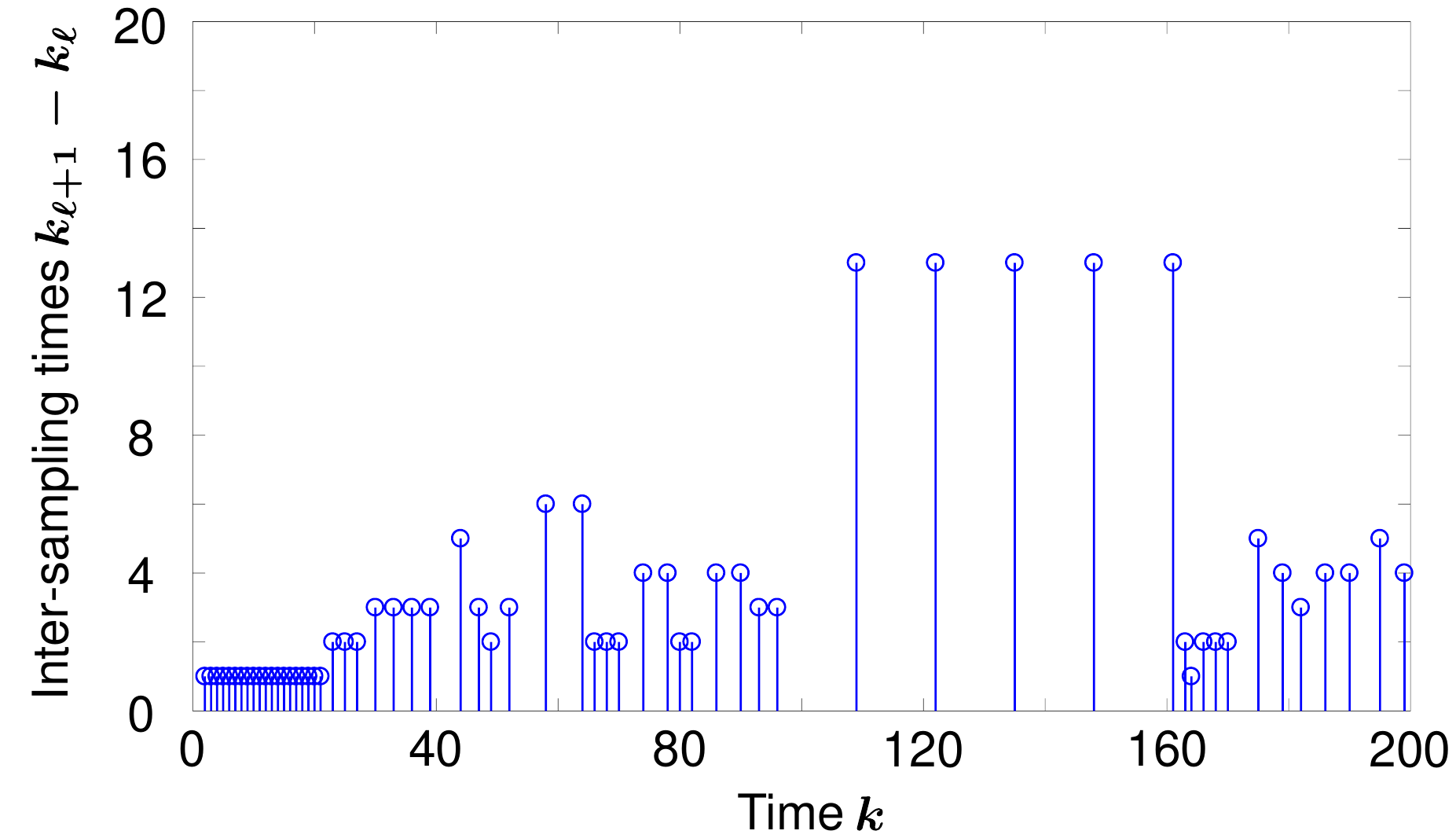}} \vspace{12pt}\\
	\subcaptionbox{Case $N=101$. \label{fig:ie_time_N101}}
	{\includegraphics[width = 8cm,clip]{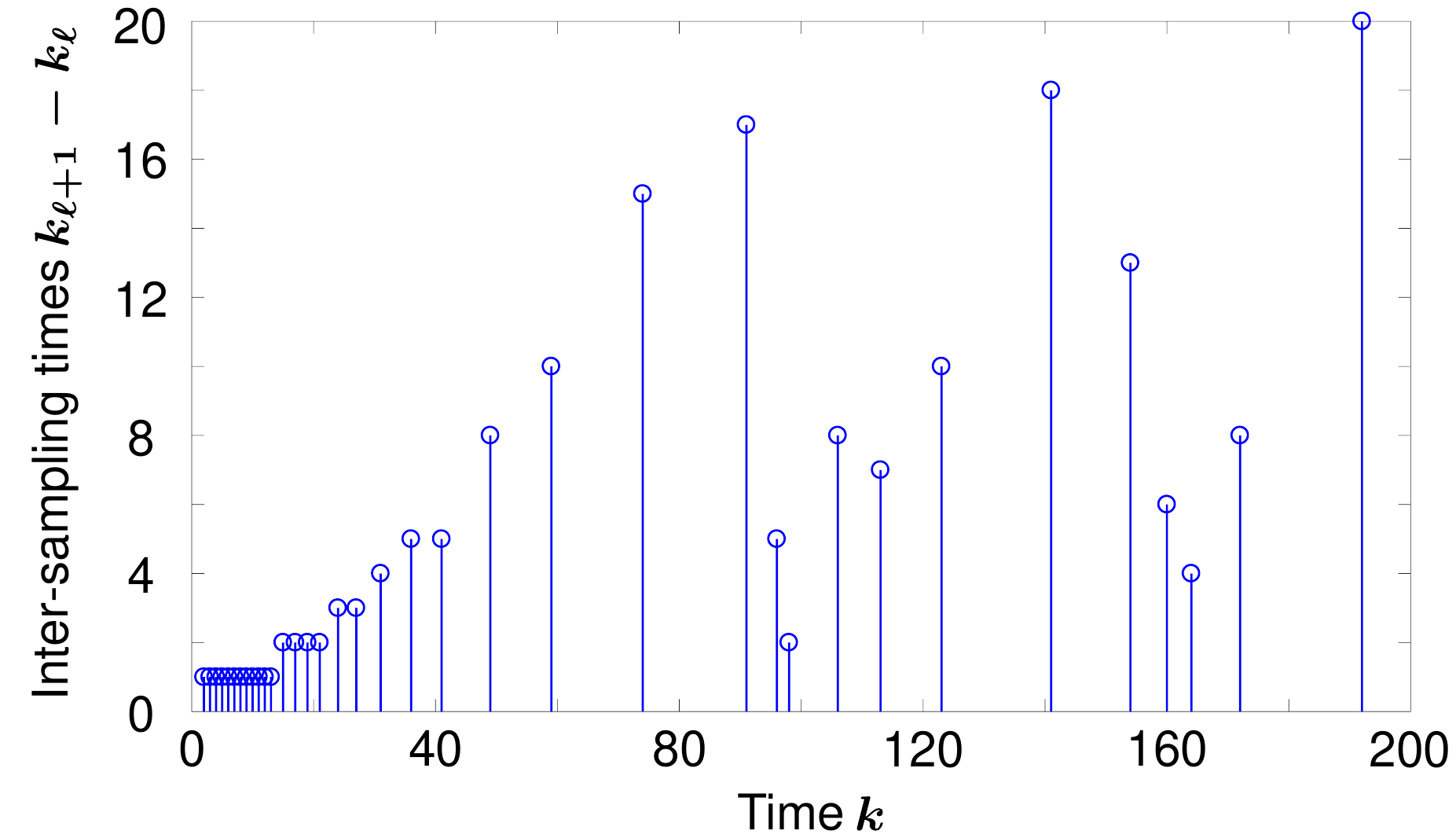}}
	\caption{Inter-sampling times $k_{\ell+1} - k_{\ell}$.}
\end{figure}

Fig.~\ref{fig:state_bound} plots the 
sequence $\{E_{\ell}\}_{\ell \in \mathbb{N}_0}$ of state bounds
used for the encoding and decoding scheme.
The blue circles and the red squares 
indicate the cases $N=61$ and $N=101$, respectively.
In both  cases, $\{E_{\ell}\}_{\ell \in \mathbb{N}_0}$
converges to zero; see also \eqref{eq:wide_E_conv}.
We have shown in the proof of Theorem~\ref{thm:main}
that the decay rate 
\[
\big(\gamma^{k_{\ell+1}-k_\ell} (1- \delta \sigma
) + \delta \sigma\big)^{\frac{1}{k_{\ell+1}-k_\ell}}
\] 
in the update rule \eqref{eq:E_def}
becomes smaller as the inter-sampling time $k_{\ell+1}-k_\ell$ increases.
Since the inter-sampling times 
in the case $N=101$ are large compared with
those in the case $N=61$ as seen in Figs.~\ref{fig:ie_time_N61}
and \ref{fig:ie_time_N101},
the convergence speed of the red squires ($N=101$) is 
slightly slower than that of the blue circles ($N=61$) 
in Fig.~\ref{fig:state_bound}.
\begin{figure}[tb]
	\centering
	\includegraphics[width = 8cm]{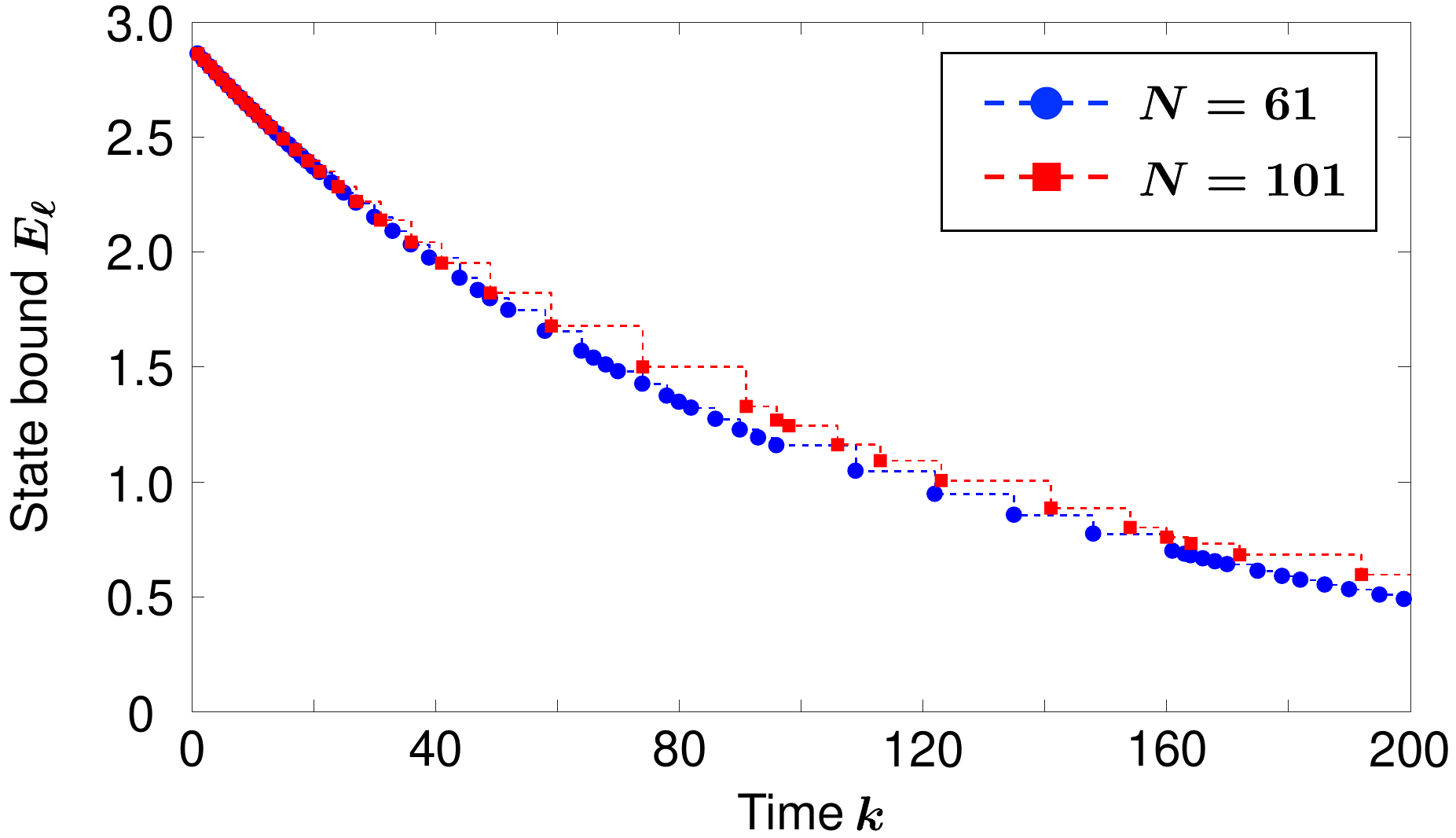}
	\caption{State bound $E_{\ell}$.}
	\label{fig:state_bound}
\end{figure}

Figs.~\ref{fig:x1_N61} and \ref{fig:x1_N101} show 
the time responses of the first element $x^{\langle 1 \rangle}(k) \in \mathbb{R}$
of the state $x(k)$ and 
its quantized value $q^{\langle 1 \rangle}_{\ell}$
in the cases $N=61$ and $N=101$, respectively.
The average values of the quantization errors are given by
$1.2523 \times 10^{-2}$ in the case $N=61$ and $9.6675 \times 10^{-3}$ in the case $N=101$.
As expected,
the quantization errors in the case $N=101$
are smaller on average than those in the case $N=61$. 
We see from Figs.~\ref{fig:x1_N61} and \ref{fig:x1_N101} that
the accurate information on the state due to 
fine quantization is utilized to reduce
the number of data transmissions.
\begin{figure}[tb]
	\centering
	\subcaptionbox{Case $N=61$. \label{fig:x1_N61}}
	{\includegraphics[width = 8cm,clip]{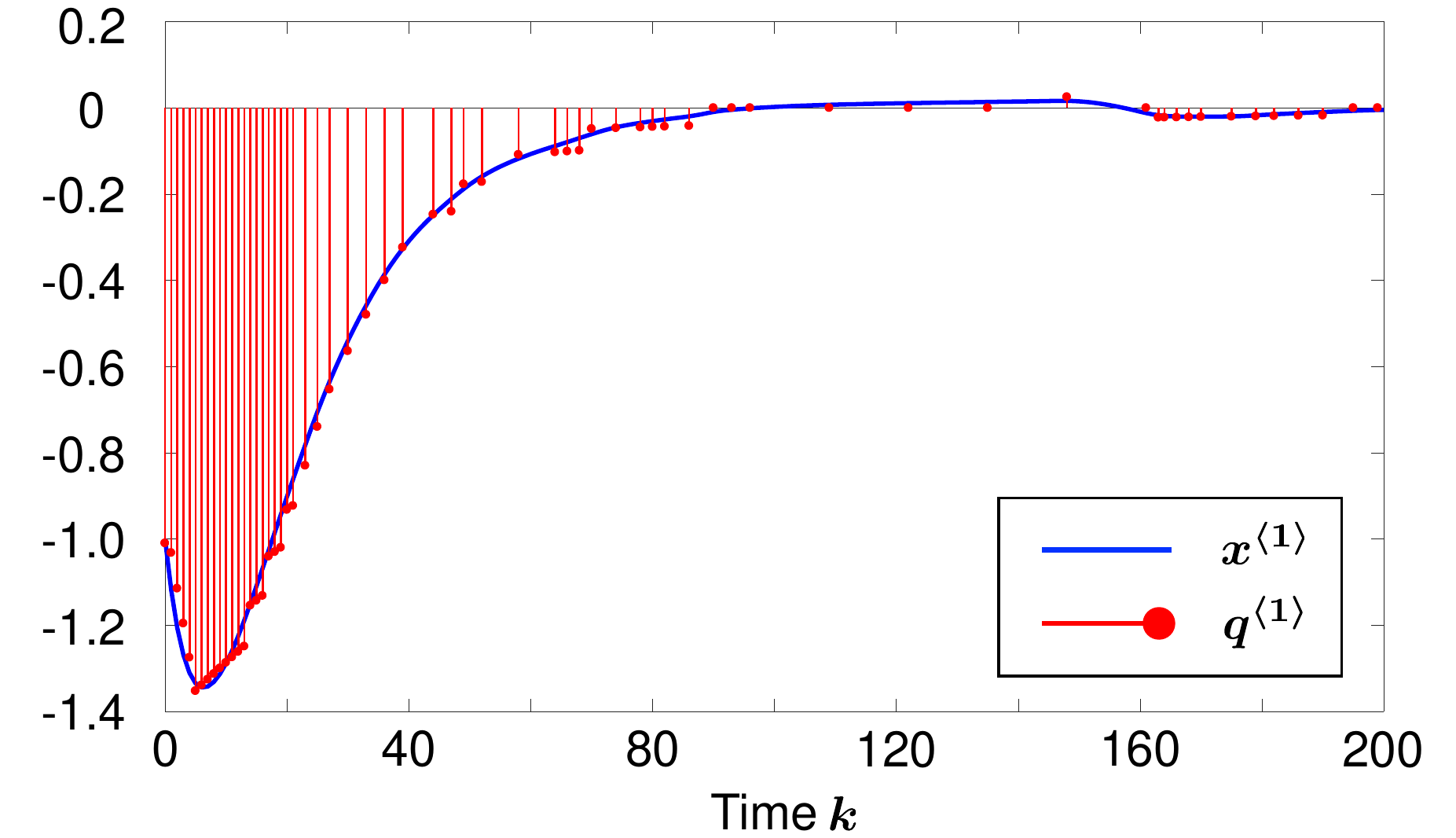}} \vspace{12pt}\\
	\subcaptionbox{Case $N=101$. \label{fig:x1_N101}}
	{\includegraphics[width = 8cm,clip]{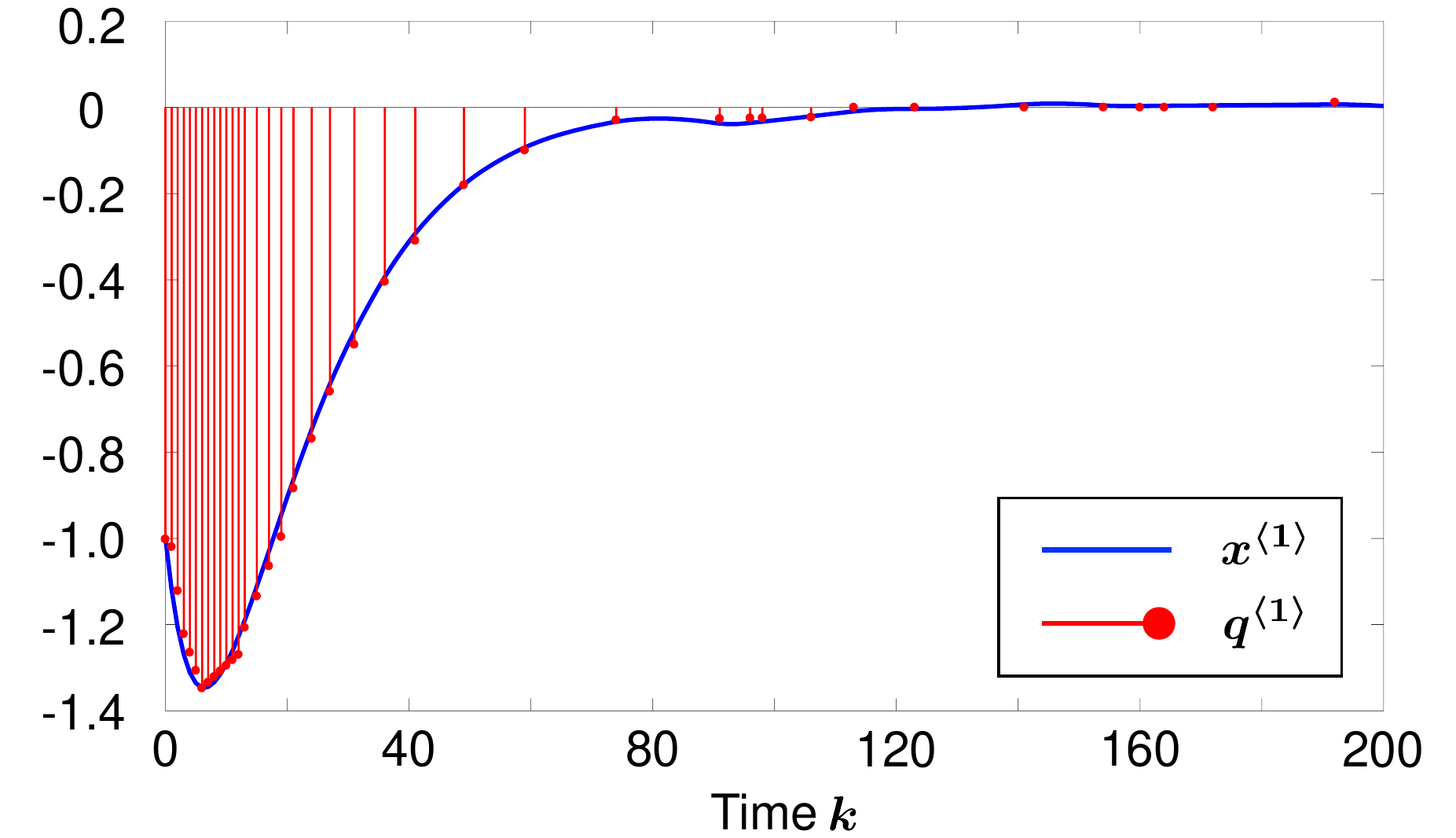}}
	\caption{State $x^{\langle 1 \rangle }$ and quantized value 
		$q^{\langle 1 \rangle }$.}
\end{figure}

\section{Conclusion}
\label{sec:conclusion}
We have developed a joint strategy of encoding and self-triggered sampling
for the stabilization of discrete-time linear systems.
The encoding method is distributed in the sense that 
an individual sensor encodes its measurements without knowing
measurement data of other sensors. 
To compute sampling times,
the centralized self-triggering mechanism 
integrates quantized measurement data sent from possibly spatially
distributed sensors
and then
estimates 
input errors due to quantization and self-triggered sampling.
We have provided a sufficient condition for the stabilization of
the quantized self-triggered control system.
This sufficient condition is described by inequalities 
on the number of quantization levels and the threshold parameter
of the self-triggering mechanism.
Future work involves extending the proposed method 
to output feedback stabilization in the presence of disturbances and
guaranteed cost control.

\section*{Appendix}
\subsection*{Proof of Lemma~\ref{lem:contraction}}
First we show that the map $\vertiii {\cdot}$ is a norm on $\mathbb{R}^n$.
Since 
\begin{equation}
\label{eq:new_norm_k_bound}
\big\|\gamma^{-k} F^k \xi \big\|_{\infty}  \leq 
\gamma^{-k} \Gamma \gamma^k \|\xi \|_{\infty}  \leq \Gamma \|\xi\|_{\infty}
\end{equation}
for all $k \in \mathbb{N}_0$ and  $\xi \in \mathbb{R}^n$,
it follows that $\vertiii {\xi} < \infty$ for all $\xi \in \mathbb{R}^n$.

By definition, $\vertiii {\xi} \geq 0$ for every 
$\xi \in \mathbb{R}^n$ and $\vertiii {0} = 0$.
Since
\begin{equation}
\label{eq:xnorm_relation}
\|\xi\|_{\infty} = 
\big\|\gamma^{-0}F^0 \xi\big\|_{\infty}  \leq \sup_{k \in \mathbb{N}_0} \big\|\gamma^{-k}F^k \xi\big\| = \vertiii {\xi}
\end{equation}
for every $\xi \in \mathbb{R}^n$, 
it follows that 
$\vertiii {\xi} = 0$ implies $\xi = 0$.
For all $a\in \mathbb{R}$ and  $\xi \in \mathbb{R}^n$,
\begin{align*}
\vertiii {a\xi} &=  
\sup_{k \in \mathbb{N}_0} \big\|\gamma^{-k}F^k (a\xi)\big\|_{\infty} \\
&= |a| 
\sup_{k \in \mathbb{N}_0} \big\|\gamma^{-k}F^k \xi\big\|_{\infty} \\
&= 
|a| ~\! \vertiii {\xi} .
\end{align*}
For every $\xi,\zeta \in \mathbb{R}^n$,
\begin{align*}
\vertiii {\xi+\zeta} &=
\sup_{k \in \mathbb{N}_0} \big\|\gamma^{-k}F^k (\xi+\zeta)\big\|_{\infty} \\
&\leq
\sup_{k \in \mathbb{N}_0} \big( \big\|\gamma^{-k}F^k \xi\big\|_{\infty} + \big\|\gamma^{-k}F^k \zeta\big\|_{\infty} \big) \\
&\leq
\sup_{k \in \mathbb{N}_0} \big\|\gamma^{-k}F^k \xi\big\|_{\infty}+ \sup_{k \in \mathbb{N}_0} \big\|\gamma^{-k}F^k \zeta\big\|_{\infty} \\
&= 
\vertiii {\xi} + \vertiii {\zeta} .
\end{align*}
Thus, $\vertiii {\cdot } $ is a norm on $\mathbb{R}^n$.

Next we prove that the norm $\vertiii {\cdot } $
has the properties \eqref{eq:prop1} and \eqref{eq:prop2}.
Take $\xi \in \mathbb{R}^n$.
We have already shown in \eqref{eq:xnorm_relation} that 
$\|\xi\|_{\infty} \leq \vertiii {\xi}$.
On the other hand, \eqref{eq:new_norm_k_bound} yields
\[
\vertiii {\xi} = \sup_{k \in \mathbb{N}_0} \big\|\gamma^{-k}F^k \xi\big\|  \leq \Gamma 
\|\xi\|_{\infty}.
\]
Therefore, \eqref{eq:prop1} holds.
The remaining assertion \eqref{eq:prop2} follows by
\begin{align*}
\bigvertiii {F^k \xi} &= 
\sup_{\ell \in \mathbb{N}_0} \big\|\gamma^{-\ell}F^{k+\ell} \xi\big\| \\
&= 
\gamma^k \sup_{\ell \in \mathbb{N}_0} \big\|\gamma^{-(k+\ell)}F^{k+\ell} \xi\big\| \\
&\leq 
\gamma^k \vertiii {\xi}\qquad \forall k \in \mathbb{N}_0.
\end{align*}
This completes the proof.
\hspace*{\fill} $\blacksquare$


\end{document}